\documentclass[11pt]{article}
\usepackage[english]{babel}
\usepackage{amssymb, amsmath, amsthm} 
\usepackage[a4paper, centering]{geometry}
\usepackage{graphicx}
\usepackage{caption}
\usepackage{subfigure,color}
\usepackage[font={small,up}]{caption}

\numberwithin{equation}{section} 

\usepackage{pgfplots}
\usepackage{bm}

\usepackage{amsthm}
\makeatletter
\def\th@plain{%
  \thm@notefont{}% same as heading font
  \itshape % body font
}
\def\th@definition{%
  \thm@notefont{}% same as heading font
  \normalfont % body font
}
\makeatother
\newtheorem{thm}{Theorem}[section]
\newtheorem{prop}[thm]{Proposition}
\newtheorem{lem}[thm]{Lemma}

\theoremstyle{definition}

\newtheorem{remark}[thm]{Remark}

\newcommand{\R}{\mathbb{R}}

\renewcommand{\epsilon}{\varepsilon}

\def\e{\epsilon}
\usepackage{enumitem}
\setlist[description]{nosep}

\title{Delayed loss of stability in singularly perturbed finite-dimensional gradient flows}

\author{{\scshape Giovanni Scilla}\\
Department of Mathematics and Applications ``R. Caccioppoli''\\ University of Naples ``Federico II''\\
Via Cintia, Monte S. Angelo - 80126 Naples \\
(ITALY)
 \\
\\
{\scshape Francesco Solombrino}\\
Department of Mathematics and Applications ``R. Caccioppoli''\\ University of Naples ``Federico II''\\
Via Cintia, Monte S. Angelo - 80126 Naples \\
(ITALY)}
\date{}

\begin{document}

\maketitle

\begin{abstract}
In this paper we study the singular vanishing-viscosity limit of a gradient flow in a finite dimensional Hilbert space, focusing on the so-called delayed loss of stability of  stationary solutions. We find a class of time-dependent energy functionals and initial conditions for which we can explicitly calculate the first discontinuity time $t^*$ of the limit. For our class of functionals, $t^*$ coincides with the blow-up time of the solutions of the %linearization of the system
linearized system around the equilibrium, and is in particular strictly greater than the time $t_c$ where strict local minimality with respect to the driving energy gets lost.  Moreover, we show that, in a right neighborhood of $t^*$, rescaled solutions of the singularly perturbed problem converge to heteroclinic solutions of the gradient flow.  Our results complement the previous ones by Zanini \cite{Zanini}, where the situation we consider was excluded by assuming the so-called transversality conditions, and the limit evolution consisted of strict local minimizers of the energy up to a negligible set of times.
\end{abstract}


\section{Introduction}

{\bf \noindent Setting of the problem}. The analysis of singularly perturbed gradient flow-type problems, in particular vanishing-viscosity limits, is often related to some physical models like as quasistatic rate-independent processes (see e.g.~\cite{MiRou} for a general introduction to the topic) and then has been widely studied in the last years (see e.g. \cite{MielkeTrusk,Negri1,Ago-Rossi,Zanini,Ago1} and the references therein).

For a small positive parameter $\epsilon$, we consider the singularly perturbed problem
\begin{equation}
\epsilon\dot{u}_\epsilon(t)=-\nabla_xF(t,u_\epsilon(t)),
\label{problem1}
\end{equation}
where $t\in[0,T]$, the driving energy $F:[0,T]\times X\to\R$ is sufficiently smooth, %and satisfies suitable assumptions (see {\bf (F0)-(F3)} in Section~\ref{prel}), 
$X$ is a finite-dimensional Hilbert space and $\nabla_x F(t,x)$ denotes the differential of $F(t,x)$ with respect to variable $x$.

When dealing with problems as in (\ref{problem1}), the first natural object is to prove that the solutions $u_\epsilon$ converge as $\epsilon\to0$, up to the extraction of a subsequence, to a limit function $u$ pointwise in $[0,T]$. Then one aims to describe the evolution of $u$, expecting $u$ to be a curve of critical points with jumps at degenerate critical points for $F(t,\cdot)$, and to characterize the behavior of $u$ at the jumps. 

\medskip

{\bf \noindent Previous results in literature}. A first contribution to these problems, in finite dimension, was provided by Zanini in~\cite{Zanini}. In that paper, the fundamental assumptions on the sufficiently smooth energy $F$ are: (i) $F$ has a finite number of critical points (ii) the vector field $\nabla_x F(t,x)$ satisfies the so called \emph{transversality conditions} at every degenerate critical point (see~\cite[Assumption 2]{Zanini}). Under these hypotheses, it was shown in~\cite[Theorem~3.7]{Zanini} that, starting from a suitable initial datum $u_0$, there exists a unique piecewise-smooth curve $u$ with a finite jump set $J=\{t_1,\dots,t_k\}$ such that:
\begin{equation}
\nabla_x F(t,u(t))=0
\label{problem2}
\end{equation}
with $\nabla_x^2F(t,u(t))$ positive definite for all $t\in[t_{i-1},t_i)$ and $i=1,\dots, k-1$. Moreover, the whole sequence $(u_\epsilon)_\epsilon$ converges to $u$ uniformly on compact subsets of $[0,T]\backslash J$. At every jump point $t_i\in J$, the left limit $u_-(t_i)$ is a degenerate critical point for $F(t_i,\cdot)$ and there exists a \emph{unique} curve $v$ connecting $u_-(t_i)$ to the right limit $u_+(t_i)$, in the sense that $\displaystyle\lim_{s\to-\infty}v(s)=u_-(t_i)$, $\displaystyle\lim_{s\to+\infty}v(s)=u_+(t_i)$, and such that
\begin{equation}
\dot{v}(s)=-\nabla_xF(t_i,v(s)), \quad \text{for all }s\in\mathbb{R}.
\end{equation}
It is also proved that suitable rescalings of $u_\epsilon$ converge to the heteroclinic solution $v$. 

A more general  point of view is taken in a recent paper by Agostiniani and Rossi~\cite{Ago-Rossi}, whose results hold for a wider class of energies, not necessarily complying with the transversality conditions. Combining ideas from the variational approach to gradient 
flows (see e.g.~\cite{AmbGiSav}) with the techniques for the vanishing-viscosity approximation of rate-independent systems (see references in \cite{Ago-Rossi}), they proved the existence of a limit curve $u$ by an argument of compactness, relying on some energetic-type estimates and the assumption that for every $t\in[0,T]$, the critical points of $F(t,\cdot)$ are \emph{isolated} (actually, this hypothesis is implied by the transversality conditions). Under the previous assumptions, they proved (see~\cite[Theorem 1]{Ago-Rossi}) that, up to a subsequence, $(u_\epsilon)_\epsilon$ pointwise converge to a solution $u$ of the limit problem (\ref{problem2}) at each continuity point $t$. Moreover, $u$ has a countable jump set and its right and left limits $u_+(t)$ and $u_-(t)$ at each jump point exist. It is worth to mention that their analysis goes beyond these results, since they provided also a suitable energetic characterization of its fast dynamics at jumps and, under some growth assumption on $F$ at its critical points (see \cite[Theorem~2]{Ago-Rossi}), they showed that $u$ improves to a \emph{balanced viscosity solution} of (\ref{problem2}), in the sense of Mielke, Rossi and Savar\'e (see e.g.~\cite{MiRoSav1,MiRoSav2}).

\medskip

{\bf \noindent Our results.} Our paper deals with a qualitative analysis of (\ref{problem1}) in the spirit of \cite{Zanini}, still in finite dimension, but focusing on a situation which is not compatible with the transversality conditions. Indeed, a direct consequence of these conditions is that whenever a curve of stationary points reaches a degenerate critical point $(\tau, \xi)$ where $\nabla^2_x F(\tau, \xi)$ has a zero eigenvalue, this one is a turning point for the curve. In particular, solutions of $\nabla_x F(t, x)=0$ exist, locally around $\xi$, only for $t\le\tau$ (see \cite[Remark 2.2]{Zanini}), which intuitively forces a jump in the limit evolution. Although transversality conditions are generic in the sense of  \cite{Ago-Rossi-Sav}, they exclude interesting situations, often appearing in the applications: for instance, stationary solutions to \eqref{problem2} whose stability changes depending on the time $t$,  usually giving rise to bifurcation of other branches of critical points.

We namely assume that the equation  $\nabla_x F(t, x)=0$  has the trivial solution $x=0$ for all $t\in[0,T]$: the equilibrium $x=0$ is however a strict local minimizer of the energy $x\mapsto F(t,x)$ only for all $t\in[0,t_c)$, with $t_c<T$, and eventually turns to a saddle point. It is not difficult to see (Proposition \ref{propzero}) that the behavior of the limit $u(t)$ of solutions to \eqref{problem1} with vanishing initial data is different than the one in \cite{Zanini}. Indeed, no istantaneous jump occurs for $t=t_c$, and we still have $u(t)=0$ in a right neighborhood of $t_c$. This fact is known as {\it delayed loss of stability}  in the literature about singular perturbation of ODE's, where it is also referred to as Ne\v{\i}shstadt phenomenon~\cite{Neishstadt1,Neishstadt2}.  The occurrence of the phenomenon is also discussed by Mielke and Truskinovsky in the framework of discrete visco-elasticity models (see~\cite{MielkeTrusk} for details). In the case of ODE's it is possible to exactly compute the first discontinuity time $t^*$ for $u(t)$ (see, for instance \cite{Russi}). Our main goal is to perform a similar task in our context. 

We assume the driving energy $F$ to satisfy the assumptions {\bf  (F0)-(F3)}, Section~\ref{prel}, considered in \cite{Ago-Rossi}, in order to recover compactness, and in Theorem \ref{jump} we explicitly characterize $t^*$ upon considering additional assumptions  {\bf (A1)-(A2)} on the Hessian matrix of the driving energy $A(t):=\nabla_x^2F(t,0)$.

We have to warn the reader that our assumptions involve significant restrictions on the class of energies we may consider. Besides assuming that the minimal eigenvalue $\lambda_1(t)$ of the matrix $A(t)$ remains simple on the time interval, we have to assume that the corresponding one-dimensional eigenspace is fixed. Under these hypotheses, we may characterize $t^*$ as the first time where the primitive function $t\mapsto \int_0^t \lambda_1(s)\,\mathrm{d}s$ changes its sign, or (equivalently within our assumptions) as the minimal time after which solutions of the linearized system
\[
\epsilon \dot u= A(t)u
\]
with suitable, generic initial data (see \eqref{assume}) blow up in the limit. We see this as a good guiding principle for further analysis; it is  however in our opinion a nontrivial issue how, and to which degree of  generality, the result can be extended without using  assumption {\bf (A1)}. Its role is apparent in the (still not straightforward) proof of Claim 1 in Theorem \ref{jump}, allowing us to conveniently estimate the effect of the nonlinear terms, in a fixed direction, locally around the equilibrium. We however remark that our assumptions only involve the linearization of \eqref{problem1} around the stationary solutions and are easy to be checked in practice (see Remark \ref{comments}).

The gradient flow structure of \eqref{problem1} is relevantly exploited in the proof of Theorem \ref{jump}. Indeed, when (according to Claim 1) the trajectories $u_\e$ reach the boundary of a ball which contains only $0$ as critical point, at $t_\epsilon\approx t^*$, a nonvanishing amount of energy is spent due to the presence of the term
\[
\int_0^{t_\epsilon} \|\dot u_\epsilon(s)\|\,\|\nabla_x F(s, u_\epsilon(s)\|\,\mathrm{d}s
\]
in the energy balance for the solutions of \eqref{problem1}. This gap in the energy leads to the formation of a discontinuity in the limit. We have indeed to exclude the possibility that,  as an effect of the nonlinear terms, trajectories may return asymptotically close to the saddle point $x=0$ in a fast time, which would result in no discontinuity in the limit. This is excluded in our case by energetic reasons. We also notice that having a-priori established compactness according to Theorem \ref{AR}, as deduced in \cite{Ago-Rossi} for singularly perturbed gradient-flows, is crucial to the argument of Theorem \ref{jump}. This is another difference with the situation in \cite{Zanini}, where compactness is recovered by the explicit construction of the limit. We also point out that in the second part of the proof, after establishing Claim 1,  assumption  {\bf (A1)} is not used, and the argument can be successfully used any time one is able to establish Claim 1.

Once the occurrence of this delayed change of stability is stated, we analyze, along the lines of~\cite[Theorem~3.5]{Zanini}, the behavior of the limit evolution $u$ in a right neighborhood of the discontinuity point $t^*$. More precisely, we prove that if $t_\e\to t^*$ is suitably chosen, the limit points of the rescaled functions $w_\epsilon(s):=u_\epsilon(t_\epsilon +\epsilon s)$ are heteroclinic solutions of the autonomous gradient system
\begin{align}\label{heterosystem0}
\dot w(s)=-\nabla_x F(t^*, w(s))
\end{align}
originating from the equilibrium $w=0$ for $s\to -\infty$. It is worth to note that, differently from the case considered in \cite{Zanini}, here there is no uniqueness of heteroclinic trajectories. Indeed, we show the existence of at least two distinct ones, depending on the sign of $\langle u_\epsilon(0),e_1\rangle$ the component along the eigenspace corresponding to the minimal eigenvalue of $\nabla^2_x F(t^*, 0)$. From a different and more general point of view, Agostiniani and Rossi (see~\cite[Proposition 3.1]{Ago-Rossi}) proved that the left and the right limits $u_-(t^*)$ and $u_+(t^*)$ of the limit evolution $u$ at a discontinuity point $t^*$ can be connected by a finite union of heteroclinic solutions of the system (\ref{heterosystem0}). However, this result does not contain our aforementioned convergence property, that can be used to determine explicitly the right limit $u_+(t^*)$ in some simple situations, like as the one where the Hessian matrix $\nabla_x^2F(t^*,0)$ has only one negative eigenvalue, the minimal one $\lambda_1(t^*)$.  This case is analysed in Section \ref{single}.

While the analysis of the present paper is confined to finite dimension, we think there are several reasons of interest for possible extensions to infinite dimensional settings. 
As a relevant example, coming from the applications, of an energy-driven evolution where a stationary configuration changes its stability along the time, and eventually evolves into different structures, we may  mention for instance the experiments in \cite{Plos}. In an intrinsically curved elastic material subject to a boundary load, the straight configuration, which is locally minimizing for large values of the force, loses stability upon slowly releasing and evolves into nontrivial structures, emerging as bifurcation branches at different critical values of the force.

\medskip

{\bf \noindent Plan of the paper.} The paper is organized as follows. In Section~\ref{prel} we fix notation and state the problem, recalling the main assumptions under which the general existence Theorem~\ref{AR} by Agostiniani and Rossi (see \cite[Theorem 1]{Ago-Rossi}) holds. Section~\ref{dellostab} deals with the main results, concerning the delayed loss of stability of the trivial equilibrium $u=0$. First, we show that if $(u_\epsilon)_\epsilon$ is a solution to the singularly perturbed problem (\ref{problem1}), such that $u_\epsilon(0)\to0$, then it converges, as $\epsilon\to0$, to $u=0$ on $[0,t^*)$ (Proposition~\ref{propzero}), with $t^*$, defined by (\ref{tstar}). Then, under assumptions {\bf (A1)-(A2)} on the Hessian matrix of the driving energy $\nabla_x^2F(t,0)$ and choosing initial data suitably decaying to 0 with polynomial rate (\ref{decay}), we prove the main result (Theorem~\ref{jump}) that $t^*$ is the first  jump point for the limit function $u$. In Section~\ref{behjump} we analyze the behavior of the limit function at the jump, showing that rescaled solutions to the singularly perturbed problem converge to heteroclinic solutions of the gradient flow (Theorem~\ref{hetero-prel}). In some simple situations, the right limit $u_+(t^*)$ can be exactly determined: this is the case, for instance, when the Hessian matrix $\nabla_x^2F(t,0)$ possesses just one negative eigenvalue (Theorem~\ref{jump2}). Finally, in Section~\ref{one-dimensional}, we revisit the well-known one-dimensional case (see e.g. \cite{Russi}) under different assumptions.

\section{Notation and preliminary results}\label{prel}

Fix $T>0$. We consider the singular limit, as $\epsilon\to0$, of the gradient flow equation
\begin{equation}
\epsilon \dot{u}=-\nabla_{{x}} F(t,u) \quad \text{in $X$ for a.e. $t\in[0,T]$,}
\label{eq1}
\end{equation}
where $(X,\|\cdot\|)$ is a $n$-dimensional Hilbert space, $n\geq1$, the energy functional $F: [0,T]\times X\to\mathbb{R}$ is sufficiently smooth, say
\\
\begin{description}
\item[(F0)] $F\in C^2([0,T]\times X)$,
\end{description}
\indent
\\
and $\nabla_x F$ denotes the differential of $F(t,x)$ with respect to the variable $x$. Moreover, we require on $F$ the following assumptions:
\\
\begin{description}
\item[(F1)] the map $\mathcal{F}:u\to \displaystyle\sup_{t\in[0, T]}|F(t,u)|$ satisfies: \\ 

$\forall\rho>0$, the sublevel set $\{u\in X:\, \mathcal{F}(u)\leq \rho\}$ is bounded;\\

\item[(F2)] there exist $C_1,C_2>0$ such that 
\begin{equation}
 |\partial_t F(t,u)|\leq C_1 F(t,u)+C_2,\quad \forall (t,u)\in[0,T]\times X,
\end{equation}
where $\partial_t F$ denotes the partial derivative of $F(t,x)$ with respect to the variable $t$;\\
\item[(F3)] for any $t\in[0,T]$, the set of critical points
\begin{equation}
C(t):=\{u\in X:\, \nabla_x F(t,u)=0\}
\end{equation}
consists of isolated points.
\end{description}
\indent
\\
Under the previous assumptions on the energy functional $F$, Agostiniani and Rossi (see~\cite{Ago-Rossi}) proved that, up to extracting time-independent subsequences, the solutions $u_\epsilon$ of the singularly perturbed problem (\ref{eq1}) converge {\it pointwise}, as $\epsilon\to0$, to a function $u$ solving
\begin{equation}
\nabla_x F(t,u(t))=0
\end{equation}
at each continuity point $t$. It is worth mentioning that, for our analysis, we only need this compactness result that they prove, among other things, in \cite[Theorem 1]{Ago-Rossi}. We now therefore only list  the properties that are needed in the rest of the paper. For a reader who is also familiar with the paper \cite{Ago-Rossi}, we remark that the proof of the statements below actually does not require the careful analysis of the so-called energy-dissipation cost that they perform, but only the first implication in \cite[Proposition 4.1]{Ago-Rossi}. 

\begin{thm}[Agostiniani-Rossi] \label{AR} Assume that {\bf (F0)-(F3)} hold. Then, up to a subsequence independent of $t$, $(u_\epsilon)_\epsilon$ solving (\ref{eq1}) converge pointwise, as $\epsilon\to0$, to a function $u:[0,T]\to X$ satisfying the following properties:
\begin{description}
\item[(i)] the jump set $J$ of $u$ is at most countable;
\item[(ii)] $u$ is continuous on $[0,T]\backslash J$, and solves
\begin{equation}
\nabla_x F(t,u(t))=0 \quad \text{in $X$ for every $t\in [0,T]\backslash J$};
\end{equation}
\item[(iii)] the left and right limits $u_{-}(t)$ and $u_{+}(t)$ exist at every $t\in(0,T)$, and so do the limits $u_{+}(0)$ and $u_{-}(T)$.
\end{description}
\end{thm}

\section{Delayed loss of stability}\label{dellostab}
We are interested in a situation where a trivial equilibrium exists at any time. We namely assume that $0\in C(t)$ for all $t\in[0,T]$; i.e.,
\begin{equation}
\nabla_x F(t,0)=0,\quad \forall t\in[0,T],
\label{eq2}
\end{equation}
and define
\begin{equation}
A(t):=\nabla^2_x F(t,0), \quad t\in[0,T].
\end{equation}
Throughout the paper, we will denote with the symbol $\lambda_1(t)$ the minimum eigenvalue of the matrix $A(t)$. We will assume that
\begin{align}\label{ipotesi_fond}
\lambda_1(0)>0 \quad \mbox{ and }\quad\int_{0}^T\lambda_1(s)\,\mathrm{d}s<0\,.
\end{align}

The first assumption implies in particular that $A(0)$ is strictly positive definite, so that $x=0$ is a strict local minimizer of the energy. On the other hand, the second one entails  that $x=0$ must have turned to a saddle point at some time in the interval $(0, T)$. We further note that, due to \eqref{ipotesi_fond},  it is well defined
\begin{equation}
t^*:=\min\left\{t\in(0,T):\, \int_{0}^t\lambda_1(s)\,\mathrm{d}s=0\right\}.
\label{tstar}
\end{equation}

We immediately notice that $t^*$ is strictly larger than the critical time $t_c$ where bifurcation for the stationary functional may occur, that is \begin{equation}
t_c:=\inf\left\{t\in(0,T):\, \lambda_1(t)=0\right\}.
\end{equation}
Indeed, as we are going to show in the next proposition, $t_c$ plays no significant role for the limit evolution. Trajectories starting close to $0$ will stay close to the trivial equilibrium at least until $t^*$.  This is a different scenario than the one considered in \cite{Zanini}, where, in view of the {\it transversality conditions} (see \cite[Assumption 2]{Zanini}), a jump in the limit trajectory occurs exactly at the time when local minimality gets lost.
\begin{prop}\label{propzero}
Let $(u_\epsilon)_\epsilon$ be a sequence of solutions to \eqref{eq1}, with $u_\epsilon(0)\to0$ as $\epsilon\to0$. Assume that \eqref{eq2} and \eqref{ipotesi_fond} hold, and define $t^*$ as in \eqref{tstar}. Then, for all $t\in[0,t^*)$, $u_\epsilon(s)\to0$ as $\epsilon\to0$ uniformly in $[0,t]$.
\end{prop}

\proof
We fix $t\in[0,t^*)$.
Since $\lambda_1(0)>0$ and, by assumption, $\int_{0}^s\lambda_1(\tau)\,\mathrm{d}\tau>0$ for all $s\in (0,t]$, by the mean value theorem we may find $\eta>0$ such that 
\begin{equation}
\int_{0}^s(\lambda_1(\tau)-\eta)\,\mathrm{d}\tau\geq0
\end{equation}
for all $s\in [0,t]$. We now set
\begin{equation}
B(t,u):=\nabla_x F(t,u)-A(t)u\,.
\end{equation}
By construction, $\|B(t,u)\|=o(\|u\|)$ as $\|u\|\to0$, uniformly with respect to $t$. Therefore, for $\eta$ fixed as above, we find
$\sigma=\sigma(\eta)>0$
with
\begin{equation}\label{B}
\|B(t,u)\|\leq \eta\|u\|, \quad \text{if }\|u\|\leq \sigma\,.
\end{equation}

We now define
\begin{equation}
t^{\sigma,\epsilon}:=\inf \left\{s\in[0,t]:\, \|u_\epsilon(s)\|\geq \sigma\right\},
\end{equation}
that is, the first time $s\in[0,t]$ such that $\|u_\epsilon(s)\|\geq \sigma$. We note that, since $u_\epsilon(0)\to0$ as $\epsilon\to0$, for $\epsilon$ small enough it results $t^{\sigma,\epsilon}>0$.

For any $s\in[0,t^{\sigma,\epsilon}]$ we have
\begin{equation}
\begin{split}
\frac{\mathrm{d}}{\mathrm{d}s}\frac{\|u_\epsilon(s)\|^2}{2}&=\langle u_\epsilon(s), \dot{u}_\epsilon(s)\rangle=-\frac{1}{\epsilon}\langle A(s) u_\epsilon(s), {u}_\epsilon(s)\rangle\\
&-\frac{1}{\epsilon}\langle B(s,u_\epsilon(s)), {u}_\epsilon(s)\rangle\\
&\leq -\frac{\lambda_1(s)}{\epsilon}\|u_\epsilon(s)\|^2+\frac{\eta}{\epsilon}\|u_\epsilon(s)\|^2\\
&=-\frac{\|u_\epsilon(s)\|^2}{\epsilon}(\lambda_1(s)-\eta).
\end{split}
\end{equation}
By Gronwall's Lemma we then deduce that
\begin{equation}
\|u_\epsilon(s)\|^2\leq \|u_\epsilon(0)\|^2 \text{exp}\left({-\frac{2}{\epsilon}\int_{0}^s(\lambda_1(\tau)-\eta)\,\mathrm{d}\tau}\right),\quad \forall s\in[0,t^{\sigma,\epsilon}].
\end{equation}
The previous estimate implies that $t^{\sigma,\epsilon}=t$ and $u_\epsilon(s)\to0$ as $\epsilon\to0$ uniformly in $[0,t]$.
\endproof

Our aim is then to provide sufficient conditions under which (at least for properly chosen initial conditions), the limit trajectory $u(t)$, whose existence is provided by the general Theorem \ref{AR},  has a jump exactly at $t^*$. To this end, we have to make some assumptions on the Hessian matrix $A(t)$ of the energy at $0$. We namely require that
\begin{itemize}
\item[{\bf (A1)}] there exists $\rho >0$ such that the minimum eigenvalue $\lambda_1(t)$ of $A(t)$ is \emph{simple} for all $t\in[0,t^*+\rho]$ and its eigenspace $L$ is fixed. We denote by $e_1$ its generator, i.e. $L=\text{span}(e_1)$;
\item[{\bf (A2)}] $\mathrm{det}(A(t^*))\neq0$. 
\end{itemize}

\begin{remark}
Condition {\bf (A2)} combined with \eqref{ipotesi_fond} implies in particular that $\lambda_1(t^*)<0$. With this, it easily follows from {\bf(A1)} that, in our setting, $t^*$ can be characterized as the minimal time after which solutions of the linear system
\[
\epsilon \dot u= A(t) u
\]
with vanishing initial data blow up in the limit as $\epsilon\to 0$. It indeed suffices to consider initial conditions satisfying \eqref{decay} and \eqref{assume} below, while blow-up before $t^*$ can be excluded arguing as in Proposition \ref{propzero}.
\end{remark}

\begin{remark}\label{comments}
The first of the two conditions in {\bf (A1)} is always satisfied whenever $\lambda_1(0)$ is simple and $t^*$ is sufficiently small. The second one is indeed a quite strong assumption, which is however met in a number of nontrivial situations. A straightforward sufficient condition for it, provided $\lambda_1(t)$ stays simple, is that $A(t)$ commutes with $A(0)$ for every $t$. This is for instance the case if $A(t)$ is a perturbation of the identity of the form $I\pm\varphi(t)B$ for some scalar function $\varphi(t)$ bounded away from $0$.
\end{remark}

We may now prove our main result, showing that $t^*$ is actually  a  jump point for the limit function $u$. To this end, we  have to assume that the initial data $u_\epsilon(0)$ converge to $0$ with polynomial decay rate, that is, there exists $\alpha>0$ %, and $c_1$, $c_2>0$ 
such that
\begin{align}\label{decay}
\lim_{\epsilon \to 0} u_\epsilon(0)=0\quad\mbox{and}\quad \liminf_{\epsilon \to 0} \frac{\|u_\epsilon(0)\|}{\epsilon^\alpha}>0\,.
\end{align}

\begin{thm}\label{jump}
Assume {\bf (F0)-(F3)}, as well as \eqref{eq2}. Let $(u_\epsilon)_\e$ be a sequence of solutions to (\ref{eq1}), with initial data $u_\epsilon(0)$ as in \eqref{decay}. Under assumption \eqref{ipotesi_fond}, define $t^*$ as in (\ref{tstar}). 
Assume that {\bf(A1)} and {\bf(A2)} hold, and that
\begin{equation}
\liminf_{\epsilon \to 0}\frac{|\langle u_\epsilon(0),e_1\rangle|^2}{\|u_\epsilon(0)\|^2}>0\,,
\label{assume}
\end{equation}
with $e_1$ as in {\bf(A1)}.  Let $u:[0,T]\to X$ be the pointwise limit  of $(u_\epsilon)_\epsilon$, as given by Theorem \ref{AR}. Then, $t^*$ is the minimal element of the jump set $J$ of the function $u$. 
\end{thm}

\proof
We already know by Proposition \ref{propzero} that $u(t)=0$ for all $t\in [0,t^*)$,  thus we only have to show that
\begin{equation}
u_+(t^*)\neq0.
\label{limdest}
\end{equation}

As a first point, since Assumption {\bf{(A2)}} gives that ${\rm det}(A(t^*))\neq0$, from the Implicit Function Theorem there exists $\xi>0$ such that for all $(t,u)\in[t^*-\xi,t^*+\xi]\times B_\xi(0)$ one has
\begin{equation}\label{solozero}
\nabla_x F(t,u(t))=0 \iff u(t)=0.
\end{equation}
We now fix $\hat{t}\in [t^*, t^*+\xi]$ and we show that the set $\left([t^*,\hat{t}]\cap\{s:\,\|u(s)\|\geq \xi\}\right)\setminus J$ is nonempty. If so, since  $\lim_{s\to(t^*)^+}u(s)$ exists by Theorem~\ref{AR},  the assertion \eqref{limdest} follows from the arbitrariness of $\hat t$, and since $J$ is a null set. 
Without loss of generality, we may assume that $\hat t \notin J$.

For $e_1$ being the eigenvector in {\bf (A1)}, we set
\begin{equation}\label{perp}
u^1_\epsilon(t):=\langle u_\epsilon(t),e_1\rangle\quad\mbox{ and}\quad u^\perp_\epsilon(t):= u_\epsilon(t) - u^1_\epsilon(t) e_1.
\end{equation}
With \eqref{assume}, we may fix $\delta>0$ such that
\begin{align}\label{assume2}
(1+\delta)|u^1_\epsilon(0)|^2\ge \|u_\epsilon(0)\|^2
\end{align}
for all $\epsilon >0$. Moreover, we denote by $\lambda^\perp(s)$ the minimum eigenvalue of $A(s)$ restricted to the orthogonal space $e_1^\perp:=\{v:\,\langle v,e_1\rangle=0\}$. Since $\lambda_1(s)$ is simple for any $s$, we may find $\eta>0$ such that 
\begin{align}\label{scarto}
\lambda^\perp(s)-\lambda_1(s)>\eta\left(\frac{(1+\delta)^\frac32+(1+\delta)}{\delta}\right)
\end{align}
for all $s\in [0, \hat{t}]$, where $\delta$ is given by \eqref{assume2}. Moreover, since $\int_{0}^{t^*}\lambda_1(s)\,\mathrm{d}s=0$ and $\lambda_1(t^*)<0$, if $\hat{t}-t^*$ is small enough we have $\int_{0}^{\hat{t}}\lambda_1(s)\,\mathrm{d}s<0$. Up to choosing $\eta$ smaller  we may also assume that 
\begin{align}\label{scarto2}
\int_{0}^{\hat{t}}(\lambda_1(s)+\eta\sqrt{1+\delta})\,\mathrm{d}s<0\,.
\end{align}
Correspondingly, we fix $\mu>0$ such that $\mu\leq \xi$ and
\begin{equation}
\|B(t,u)\|\leq\eta\|u\|,\quad \text{if }\|u\|\leq \mu
\end{equation}
where $B(t,u)$ is defined as in \eqref{B}. Notice that $\eta$, and thus $\mu$, actually depend on the chosen $\hat t$, but we omit to explicitly stress this dependence for the ease of notation.

We now prove the following Claim.\\
{\bf Claim 1} \emph{For $\epsilon$ small enough, the set $
R_\epsilon:=\{s\in [0, \hat{t}]:\, \|u_\epsilon(s)\|>\mu\}
$
is nonempty.}

\smallskip

\emph{Proof of Claim 1.} We argue by contradiction and we assume that $R_\epsilon$ is empty; this means that $\|u_\epsilon(s)\|\leq\mu$ for all $s\in [0,\hat{t}]$. 

Then, we set
\begin{equation}
g_\epsilon(s):=|u_\epsilon^1(s)|^2-\frac{1}{1+\delta}\|u_\epsilon(s)\|^2,
\label{geps}
\end{equation} 
and prove that $g_\epsilon(s)>0$ for all $s\in[0, \hat{t}]$. Indeed, we have $g_\epsilon(0)>0$ by (\ref{assume}). Now, if the set $\{s:\, g_\epsilon(s)=0\}$ were nonempty, by compactness it would admit a minimum $s_\epsilon$. Then we would have $g_\epsilon(s_\epsilon)=0$ and $\dot{g}_\epsilon(s_\epsilon)\leq0$. Moreover, $u_\epsilon^1(s_\epsilon)\neq0$, otherwise $g_\epsilon(s_\epsilon)=0$ would imply $u_\epsilon(s_\epsilon)=0$, which is forbidden, since $u=0$ is a stationary solution of \eqref{eq1}.
Now, it holds
\begin{equation}
\begin{split}
\frac{\mathrm{d}}{\mathrm{d}s}|u_\epsilon^1(s)|^2&=2 u_\epsilon^1(s)\dot{u}_\epsilon^1(s)=2 u_\epsilon^1(s)\langle \dot{u}_\epsilon(s),e_1\rangle\\
&=2 u_\epsilon^1(s) \left\langle -\frac{1}{\epsilon}A(s)e_1,u_\epsilon(s)\right\rangle\\
&-\frac{2}{\epsilon}u_\epsilon^1(s) \langle B(s,u_\epsilon(s)),e_1\rangle\\
&\geq  -\frac{2}{\epsilon} \left[\lambda_1(s)|u_\epsilon^1(s)|^2+\eta\|u_\epsilon(s)\||u_\epsilon^1(s)|\right].
\end{split}
\label{stim}
\end{equation}
On the other hand,
\[
\begin{split}
-\frac{\mathrm{d}}{\mathrm{d}s}\frac{\|u_\epsilon(s)\|^2}{1+\delta}&=-\frac{2}{1+\delta}\langle u_\epsilon(s), \dot{u}_\epsilon(s)\rangle\\
&=\frac{2}{\epsilon(1+\delta)}\langle A(s) u_\epsilon(s), u_\epsilon(s)\rangle + \frac{2}{\epsilon(1+\delta)}\langle B(s,u_\epsilon(s)), u_\epsilon(s)\rangle\\
&\geq \frac{2}{\epsilon(1+\delta)}\lambda_1(s) |u_\epsilon^1(s)|^2+\frac{2}{\epsilon(1+\delta)}\lambda^\perp(s) \|u_\epsilon^\perp(s)\|^2-\frac{2\eta}{\epsilon(1+\delta)}\|u_\epsilon(s)\|^2.
\end{split}
\]
We then have
\[
\begin{split}
\frac{\mathrm{d}}{\mathrm{d}s}g_\epsilon(s)&\geq \frac{2}{\epsilon(1+\delta)}\lambda^\perp(s)\|u_\epsilon^\perp(s)\|^2-\frac{2\delta}{\epsilon(1+\delta)}\lambda_1(s)|u_\epsilon^1(s)|^2\\
&-\frac{2\eta}{\epsilon(1+\delta)}\left[\|u_\epsilon(s)\|^2+(1+\delta)\|u_\epsilon(s)\||u_\epsilon^1(s)|\right].
\end{split}
\]
From the definition of $s_\epsilon$, \eqref{perp}, and \eqref{geps}  we deduce that
\[
\frac{\mathrm{d}}{\mathrm{d}s}g_\epsilon(s)\geq \frac{2}{\epsilon}\frac{\delta}{1+\delta}\left[\lambda^\perp(s_\epsilon)-\lambda_1(s_\epsilon)-\eta\left(\frac{(1+\delta)^\frac32+(1+\delta)}{\delta}\right)\right]|u_\epsilon^1(s_\epsilon)|^2>0\,,
\]
where the last inequality follows by  \eqref{scarto}. This contradicts the existence of $s_\e$.

Thus,  $g_\epsilon(s)>0$ for all $s\in[0, \hat{t}]$, that is
\begin{equation}
\|u_\epsilon(s)\|^2<(1+\delta)|u_\epsilon^1(s)|^2\,.
\label{sameproof}
\end{equation}
Inserting into equation \eqref{stim} we have
\begin{equation}
\frac{\mathrm{d}}{\mathrm{d}s} |u_\epsilon^1(s)|^2\geq -\frac{2}{\epsilon}(\lambda_1(s)+\eta\sqrt{1+\delta})|u_\epsilon^1(s)|^2.
\end{equation}
By Gronwall's Lemma we finally deduce
\begin{equation}
|u_\epsilon^1(\hat{t})|^2\geq |u_\epsilon^1(0)|^2 \text{exp}\left({-\frac{2}{\epsilon}\int_{0}^{\hat{t}}(\lambda_1(s)+\eta\sqrt{1+\delta})\,\mathrm{d}s}\right)\,.
\end{equation}
The right-hand side is now unbounded as $\epsilon \to 0$, because of \eqref{decay}, \eqref{assume} and \eqref{scarto2}, thus giving a contradiction. This concludes the proof of Claim 1.

\smallskip

Going back to the proof of Theorem~\ref{jump},  we fix $\bar{t}<t^*$ arbitrarily. Proposition \ref{propzero} yields that $\|u_\epsilon(s)\|<\mu/2$ for all $s\in [0, \bar{t}]$ when $\epsilon$ is small enough. Combining with Claim 1, we then get that there exists ${t}^2_\epsilon\in[\bar{t},\hat{t}]$ such that $\|u_\epsilon({t}^2_\epsilon)\|=\mu$.  We can also find ${t}^1_\epsilon$ such that $\|u_\epsilon({t}^1_\epsilon)\|=\mu/2$ and $\mu/2\leq\|u_\epsilon({t})\|\leq\mu$ for any $t\in[{t}^1_\epsilon,{t}^2_\epsilon]$. Since $\mu\leq\xi$, \eqref{solozero} implies that
\[
0<G_\mu:=\min\left\{\|\nabla_x F(t,u)\|:\, t\in[\bar{t},\hat{t}], \, \mu/2\leq\|u\|\leq\mu\right\}\,.
\]
From the chain rule and \eqref{eq1} we have,
\begin{equation*}
\begin{split}
F(\hat{t}, u_\epsilon(\hat{t})) - F(\bar{t}, u_\epsilon(\bar{t}))&=\int_{\bar{t}}^{\hat{t}} \frac{\mathrm{d}}{\mathrm{d}s}F(s,u_\epsilon(s))\,\mathrm{d}s\\
&= \int_{\bar{t}}^{\hat{t}} \langle \dot{u}_\epsilon(s), \nabla_x F(s, u_\epsilon(s))\rangle\,\mathrm{d}s + \int_{\bar{t}}^{\hat{t}} \partial_s F(s, u_\epsilon(s))\,\mathrm{d}s\\
&= - \int_{\bar{t}}^{\hat{t}} \|\dot{u}_\epsilon(s)\|\,\|\nabla_x F(s,u_\epsilon(s))\|\,\mathrm{d}s + \int_{\bar{t}}^{\hat{t}}\partial_s F(s, u_\epsilon(s))\,\mathrm{d}s\,.
\end{split}
\end{equation*}
On the other hand,
\begin{equation*}
\begin{split}
\int_{\bar{t}}^{\hat{t}} \|\dot{u}_\epsilon(s)\|\,\|\nabla_x F(s,u_\epsilon(s))\|\,\mathrm{d}s &\geq \int_{{t}^1_\epsilon}^{{t}^2_\epsilon} \|\dot{u}_\epsilon(s)\|\,\|\nabla_x F(s,u_\epsilon(s))\|\,\mathrm{d}s\\
&\geq G_\mu \int_{{t}^1_\epsilon}^{{t}^2_\epsilon} \|\dot{u}_\epsilon(s)\|\,\mathrm{d}s\geq G_\mu \|{u}_\epsilon({t}^2_\epsilon)-u_\epsilon({t}^1_\epsilon)\|\\
&\geq \frac{\mu G_\mu}{2}.
\end{split}
\end{equation*}
From this we deduce that
\begin{equation*}
F(\hat{t}, u_\epsilon(\hat{t}))+\frac{\mu G_\mu}{2}\leq F(\bar{t}, u_\epsilon(\bar{t}))+\int_{\bar{t}}^{\hat{t}}\partial_s F(s, u_\epsilon(s))\,\mathrm{d}s,
\end{equation*}
and passing to the limit as $\epsilon\to0$ (note that $\mu$ does not depend on $\epsilon$), we get
\begin{equation}\label{bilancio}
F(\hat{t}, u(\hat{t}))+\frac{\mu G_\mu}{2}\leq F(\bar{t}, u(\bar{t}))+\int_{\bar{t}}^{\hat{t}}\partial_s F(s, u(s))\,\mathrm{d}s.
\end{equation}

Assume now by contradiction that $\left([t^*,\hat{t}]\cap\{s:\,\|u(s)\|\geq \xi\}\right)\setminus J$ is empty. Then, using (ii) in Theorem \ref{AR}, \eqref{solozero}, and since $J$ is a null set, we have  that $u(t)=0$ for almost every $t\in[t^*,\hat{t}]$. Since the function $u$ is continuous at $t=\hat{t}$, again (ii) in Theorem \ref{AR} and \eqref{solozero} entail that $u(\hat{t})=0$. We also know that $u(t)=0$ in $[\bar{t},t^*)$  by Proposition~\ref{propzero}. Inserting in \eqref{bilancio}, we get
\begin{equation*}
F(\hat{t},0)+\frac{\mu G_\mu}{2}\leq F(\bar{t},0)+\int_{\bar{t}}^{\hat{t}}\partial_s F(s, 0)\,\mathrm{d}s\,,
\end{equation*}
which is a contradiction, since $G_\mu>0$. This concludes the proof.
\endproof

\begin{remark}
Besides providing the necessary compactness according to Theorem \ref{AR}, the gradient flow structure of our system plays another important role in the above proof. Indeed, on the one hand proving that the trajectories move away from $0$ at some time $t_\e$ close to $t^*$, along the unstable direction $e_1$, follows by a local analysis, considering \eqref{eq1} as a small perturbation of the linear system $\e \dot u= A(t)u$. On the other hand, once the nonlinear terms become relevant, they could in  general push again, in an infinitesimal amount of time, the trajectories close to the trivial equilibrium along some stable direction. This can be excluded in the case of a gradient vector field, due to energetic reasons.
\end{remark}

\section{Behavior at the jump point}\label{behjump}
\subsection{General  results}
Concerning the limit behavior of the evolution close to the discontinuity point $t^*$, we now prove a result which is similar in spirit to \cite[Lemma 4.3]{Zanini}. Namely, if we blow up the time around $t^*$, the limit points of the rescaled functions $w_\epsilon(s):=u_\epsilon(t_\epsilon +\epsilon s)$, where $t_\e\to t^*$ is suitably chosen, are heteroclinic solutions of the autonomous gradient system
\begin{align}\label{heterosystem}
\dot w(s)=-\nabla_x F(t^*, w(s))\,,
\end{align}
originating from the equilibrium $w=0$ for $s\to -\infty$. Notice that, differently from the case considered in \cite{Zanini}, here there is no uniqueness of heteroclinic trajectories, and indeed our statement contains the existence of at least two distinct ones, depending on the sign of the component along the eigenspace corresponding to the minimal eigenvalue of $\nabla^2_x F(t^*, 0)$ . 
A slightly different, more general point of view is the one of \cite[Proposition 3.1]{Ago-Rossi}, where it is shown that the left and the right limits $u_-(t)$ and $u_+(t)$ of the limit evolution $u$ at a discontinuity point $t$ can be connected by a finite union of heteroclinic solutions of the system $\dot w(s)=-\nabla_x F(t, w(s))$. Such a result, however, does not contain the convergence property stated in Proposition \ref{hetero-prel} below, which can be used in some simple situations to exactly determine $u_+(t^*)$, as we are going to discuss in the next subsection.

\begin{prop}\label{hetero-prel}
Let the assumptions of Theorem \ref{jump} be given. Then, there exists a sequence $t_\e\to  t^*$ such that, setting $w_\epsilon(s):=u_\epsilon(t_\epsilon +\epsilon s)$,  $(w_\epsilon)_\epsilon$ has a subsequence converging uniformly on the compact subsets of $\mathbb{R}$ to a solution $w$ of the problem
\[
\begin{cases}
\dot{w}(s)=-\nabla_x F(t^*,w(s))\\
\displaystyle\lim_{s\to -\infty}w(s)=0\,.
\end{cases}
\]
Furthermore, for $e_1$ as in {\bf(A1)}, if $\langle u_\epsilon(0), e_1\rangle >0$ we have
\begin{equation}\label{weak-hetero1}
\liminf_{s\to-\infty}\frac{\langle w(s), e_1\rangle}{\|w(s)\|}>0\,.
\end{equation}
If instead $\langle u_\epsilon(0), e_1\rangle <0$, then we have
\begin{equation}\label{weak-hetero2}
\limsup_{s\to-\infty}\frac{\langle w(s), e_1\rangle}{\|w(s)\|}<0\,.
\end{equation}
\end{prop}

\begin{proof}
Using {\bf (A1)} and {\bf(A2)} we may fix $\theta >0$ with
\begin{equation}
\theta\leq\frac{|\lambda_1(t^*)|}{2\sqrt{1+\delta}}\quad \hbox{and}\quad \theta<\frac{\delta}{(1+\delta)^\frac32+(1+\delta)}(\lambda_i(t)-\lambda_1(t))
\label{ipo1}
\end{equation}
for all $t\in[0,t^*+\rho]$ and $i\geq2$, where $\delta$ is given by \eqref{assume2}. Correspondingly, we may find $\mu<\|u_+(t^*)\|$ such that, for $B(t,u)$ defined as in \eqref{B}, it holds
\begin{align}\label{B_stim}
\|B(t,u)\|\leq\theta\|u\|
\end{align}
for all $u\in X$ with $ \|u\|\le\mu$.

We now set 
\begin{equation}
t_\epsilon:=\min\{t\in[0,T]:\, \|u_\epsilon(t)\|=\mu\}\,.
\label{tepsilon}
\end{equation}   
From Proposition \ref{propzero} we have $\liminf_{\epsilon \to 0} t_\epsilon \ge t^*$. Since $\mu<\|u_+(t^*)\|$ and with the pointwise convergence of $u_\epsilon(t)$ to $u(t)$, we have that  $\limsup_{\epsilon \to 0} t_\epsilon \le t^*$. Therefore, $t_\epsilon\to t^*$ as $\epsilon\to0$. Moreover, since $\|u_\epsilon(t)\|\leq \mu$ in $[0, t_\epsilon]$, using \eqref{ipo1}, \eqref{B_stim} and arguing as in the proof of (\ref{sameproof}), we get
\begin{equation}
\|u_\epsilon(t)\|^2<(1+\delta)|\langle u_\epsilon(t), e_1\rangle|^2
\label{stimaI}
\end{equation}
for all $t\in[0,t_\epsilon]$.

We define
$
w_\epsilon(s):=u_\epsilon(t_\epsilon+\epsilon s),
$
and we note that $w_\epsilon(s)$ solves the problem
\begin{equation}\label{eps-probl}
\begin{cases}
\dot{w}_\epsilon(s)=-\nabla_x F(t_\epsilon+\epsilon s, w_\epsilon(s))\\
w_\epsilon(0)=u_\epsilon(t_\epsilon).
\end{cases}
\end{equation}
Furthermore, by the definition of $t_\epsilon$ and with \eqref{stimaI} we have
\begin{align}\label{stimeII}
\|w_\epsilon(s)\|\le \mu\quad\mbox{and}\quad \|w_\epsilon(s)\|^2<(1+\delta)|\langle w_\epsilon(s), e_1\rangle|^2
\end{align}
for all $s\in \left[-\frac{t_\epsilon}\epsilon, 0\right]$. In particular we deduce
\begin{align}\label{segnocostante}
\langle u_\epsilon(0), e_1\rangle >0 \Longrightarrow \langle w_\epsilon(s), e_1\rangle >0\quad\mbox{and}\quad\langle u_\epsilon(0), e_1\rangle <0 \Longrightarrow \langle w_\epsilon(s), e_1\rangle <0\,,
\end{align}
respectively, for all $s\in \left[-\frac{t_\epsilon}\epsilon, 0\right]$.
Since $w_\e(s)$ is equibounded, from {\bf(F0)} and the Ascoli-Arzel\`a Theorem $w_\epsilon(s)$ converges, up to a subsequence, pointwise as $\epsilon\to0$ (and uniformly on the compact subsets of $\mathbb{R}$) to the solution to
\[
\begin{cases}
\dot{w}(s)=-\nabla_xF(t^*,w(s))\\
w(0)=w_0\,,
\end{cases}
\]
where $w_0$ is a limit point of $w_\epsilon(0)$. Since by construction $\|w_\epsilon(0)\|=\mu$, we have $w_0\neq 0$ and by uniqueness $w(s)\neq 0$ for all $s\in \mathbb{R}$. Moreover, \eqref{stimeII} gives
\begin{align}\label{stimeIII}
\|w(s)\|\le \mu\quad\mbox{and}\quad \|w(s)\|^2\le(1+\delta)|\langle w(s), e_1\rangle|^2
\end{align}
for all negative times $s\in (-\infty, 0]$.
Since $w(s)\neq 0$, from the second inequality above and \eqref{segnocostante} we immediately deduce \eqref{weak-hetero1} and \eqref{weak-hetero2}.

We are left to show that $\lim_{s\to -\infty}w(s)=0$. We set $w^1(s):=\langle w(s), e_1\rangle$, and we only consider the case where $w^1(s)>0$ for all $s \le 0$, the other one being similar.
For any $s\leq0$, from \eqref{B_stim} and \eqref{stimeIII} we get
\[
\begin{split}
\dot{w}^1(s)&=-\lambda_1(t^*)w^1(s)+\langle B(t,w(s)), e_1\rangle=|\lambda_1(t^*)|w^1(s)+\langle B(t,w(s)), e_1\rangle\\
                  &\geq |\lambda_1(t^*)|w^1(s)-\theta\|w(s)\|\geq |\lambda_1(t^*)|w^1(s)-\frac{|\lambda_1(t^*)|}{2\sqrt{1+\delta}}\|w(s)\|\\
                  &\geq \frac12|\lambda_1(t^*)|w^1(s)\,.
\end{split}
\]
This implies that $\displaystyle\lim_{s\to-\infty}w^1(s)=0$, which combined with \eqref{stimeIII} finally gives
$
\displaystyle\lim_{s\to-\infty}\|w(s)\|=0,
$
as desired.

\end{proof}

\subsection{The case of a single negative eigenvalue}\label{single}
A simple situation, where Proposition \ref{hetero-prel} can be used to determine the right limit $u_+(t^*)$, is the one where the Hessian matrix $A(t^*)$ has only one negative eigenvalue, namely the minimal one $\lambda_1(t^*)$. In this case, indeed, there are exactly two solutions (up to time-translations) to \eqref{heterosystem}, each leaving $w=0$ for $s\to -\infty$ from a different side of the one-dimensional unstable manifold for $w=0$, which has ${\rm span}(e_1)$ as tangent space at $0$.  With this and Proposition \ref{hetero-prel},  the sign of  $\langle u_\epsilon(0), e_1\rangle$ uniquely determines  the limit $w(s)$ of the rescaled trajectories $w_\epsilon(s)$ introduced in Proposition \ref{hetero-prel}. The corresponding $\omega$-limit point is then a good candidate for being $u_+(t^*)$, and we are going to show that this is the case, provided it is a {\it strong local minimizer} of the energy.

Before giving precise statements, we recall the following multiplicity result, which is a direct consequence of the Stable Manifold Theorem (see, e.g.\ \cite[Theorems 9.3 and 9.5]{Tes}). 
\begin{prop}\label{hetero}
Consider $F\in C^2([0,T]\times X)$ satisfying {\bf (F1)}, and assume that \eqref{eq2} holds. For a given $t^* \in [0,T]$, set $A(t^*):=\nabla^2_x F(t^*, 0)$ and assume that its ordered eigenvalues $\lambda_1(t^*),\dots, \lambda_n(t^*)$ satisfy
\begin{align}\label{autovalo}
\lambda_1(t^*)<0<\lambda_2(t^*)\le\dots\le\lambda_n(t^*)\,.
\end{align}
Let $e_1$ be a fixed unitary eigenvector corresponding to the minimal eigenvalue $\lambda_1(t^*)$. Then there exist exactly two (up to time-translations) solutions $\underline{w}$ and $\overline{w}$ to the problem
\[
\begin{cases}
\dot{w}(s)=-\nabla_xF(t^*,w(s))\\
\displaystyle\lim_{s\to -\infty}w(s)=0
\end{cases}
\]
which satisfy
\begin{equation}\label{uniq}
\displaystyle
\lim_{s\to -\infty}\frac{\langle\overline{w}(s), e_1\rangle}{\|\overline{w}(s)\|}=1,\quad\mbox{and}\quad\lim_{s\to -\infty}\frac{\langle\underline{w}(s), e_1\rangle}{\|\underline{w}(s)\|}=-1\,,
\end{equation}
respectively.
\end{prop}

Combining Propositions \ref{hetero-prel} and \ref{hetero}, we get the following Lemma. In the statement, for the two heteroclinic trajectories $\underline{w}$ and $\overline{w}$ given by Proposition \ref{hetero}, we denote with $\underline{u}^*$, and let $\overline{u}^*$ their respective $\omega$-limit points, that is
\begin{equation}\label{omega}
\underline{u}^*=\lim_{s\to +\infty} \underline{w}(s)\quad\mbox{and}\quad\overline{u}^*=\lim_{s\to +\infty} \overline{w}(s)\,.
\end{equation}
The limit points  $\underline{u}^*$ and $\overline{u}^*$ are indeed uniquely determined, since the $\omega$-limit set of bounded orbits of a gradient flow  is a connected subset of the set of stationary points (see, e.g. \cite[Lemma 6.6]{Tes} and \cite[Theorem 14.17]{Hale}) and we are assuming that they are isolated by {\bf (F3)}.

\begin{lem}
Assume {\bf (F0)-(F3)}, as well as \eqref{eq2}. Let $(u_\epsilon)_\e$ be a sequence of solutions to (\ref{eq1}), with initial data $u_\epsilon(0)$ as in \eqref{decay}. 
Assume that \eqref{ipotesi_fond}, {\bf(A1)},  {\bf(A2)}, \eqref{assume} and \eqref{autovalo} hold, for $t^*$ as in \eqref{tstar} . Fix a unitary eigenvector $e_1$ of the minimal eigenvalue $\lambda_1(t^*)$, and $\underline{u}^*$, $\overline{u}^*$ as in \eqref{omega}.
Then:
\begin{itemize}
\item
if $\langle u_\epsilon(0), e_1\rangle >0$, for any fixed $\eta>0$  there exists a sequence $\{t_\epsilon^\eta\}$ such that $t_\epsilon^\eta \to t^*$ as $\epsilon\to0$ and
\begin{equation}
\liminf_{\epsilon\to 0}\|u_\epsilon(t_\epsilon^\eta)-\overline{u}^*\|<\eta\,;
\label{asserto}
\end{equation}
\item if $\langle u_\epsilon(0), e_1\rangle <0$ an analogous statement holds with $\underline{u}^*$ in place of $\overline{u}^*$.
\end{itemize}
\end{lem}
\begin{proof}
We only consider the case $\langle u_\epsilon(0), e_1\rangle >0$. For $t_\epsilon$ as in \eqref{tepsilon}, from Propositions \ref{hetero-prel} and \ref{hetero}, using \eqref{weak-hetero1} and the first equality in \eqref{uniq}, we get that (at least along a subsequence), the functions $s\mapsto u_\epsilon(t_\epsilon+\epsilon s)$ converge, uniformly on the compact subsets of $\mathbb{R}$ to the function $\overline{w}(s)$, or a time-translation thereof. Assuming without loss of generality that $\overline{w}(s)$ is the limit, for any fixed $\eta>0$, we know that there exists  $\bar{s}>0$ such that
\[
\|\overline{w}(\bar{s})-\overline{u}^*\|<\eta,
\]
which implies
\[
\liminf_{\epsilon\to0}\|u_\epsilon(t_\epsilon+\epsilon\bar{s})-\overline{u}^*\|<\eta
\]
by the previously stated convergence. The assertion follows choosing $t_\epsilon^\eta=t_\epsilon+\epsilon\bar{s}$, since $t_\epsilon \to t^*$ by Proposition \ref{hetero-prel}.
\end{proof}

We can now state and prove the announced result, which allows to identify $u_+(t^*)$ under the assumptions \eqref{autovalo} and \eqref{locmin1} (or \eqref{locmin2}) below.

\begin{thm}\label{jump2}
Assume {\bf (F0)-(F3)}, as well as \eqref{eq2}. Let $(u_\epsilon)_\e$ be a sequence of solutions to (\ref{eq1}), with initial data $u_\epsilon(0)$ as in \eqref{decay}.  Let $u:[0,T]\to X$ be the pointwise limit  of $(u_\epsilon)_\epsilon$, as given by Theorem \ref{AR}. 
Assume that \eqref{ipotesi_fond}, {\bf(A1)},  {\bf(A2)}, \eqref{assume}, and \eqref{autovalo} hold, for $t^*$ as in \eqref{tstar} . Fix a unitary eigenvector $e_1$ of the minimal eigenvalue $\lambda_1(t^*)$, and $\underline{u}^*$, $\overline{u}^*$ as in \eqref{omega}.
Then:
\begin{itemize}
\item if 
\begin{equation}\label{locmin1}
\nabla^2_x F(t^*, \overline{u}^*) >0\quad\mbox{and}\quad \langle u_\epsilon(0), e_1\rangle >0\,,
\end{equation}
it holds $u_+(t^*)=\overline{u}^*$. Furthermore, $u$ is of class $C^2$ in a right neighborhood of $t^*$.
\item if 
\begin{equation}\label{locmin2}
\nabla^2_x F(t^*, \underline{u}^*) >0\quad\mbox{and}\quad \langle u_\epsilon(0), e_1\rangle <0\,,
\end{equation}
it holds $u_+(t^*)=\underline{u}^*$. Furthermore, $u$ is of class $C^2$ in a right neighborhood of $t^*$.
\end{itemize}
\end{thm}

\proof  
We prove the statement under Assumption \eqref{locmin1}, the proof of the other case being totally analogous.

Since $\nabla_x F(t^*,\overline{u}^*)=0$, from the Implicit Function Theorem and \eqref{locmin1} there exists a function $\varphi(t)\in C^2([t^*-\sigma,t^*+\sigma])$, with $\varphi(t^*)=\overline{u}^*$ such that $\nabla_x F(t,\varphi(t))=0$, and  $\nabla^2_x F(t,\varphi(t))>0$ for any $t\in[t^*-\sigma,t^*+\sigma]$.  In particular, there exist positive constants $\delta$ and $\lambda$ such that
\begin{align}\label{posdef}
\langle \nabla_x^2 F(t,\psi) v, v\rangle\geq \lambda\|v\|^2
\end{align}
for all $t\in [t^*-\sigma,t^*+\sigma]$ and all $\psi \in X$ with $\|\psi- \varphi(t)\|\le \delta$.
We now fix an arbitrary $0<\eta<\delta$. Then, \eqref{asserto} provides a sequence $t^\eta_\epsilon\to t^*$ such that, for $\epsilon$ small enough,
\begin{equation}
\|u_\epsilon(t^\eta_\epsilon)-\varphi(t^*)\|=\|u_\epsilon(t^\eta_\epsilon)-u^*\|<\eta.
\label{smallnorm}
\end{equation}
Up to taking a subsequence we may indeed suppose that the $\liminf$ in \eqref{asserto} is a limit. 

We now set $v_\epsilon(t):=u_\epsilon(t)-\varphi(t)$. For every $t\in[t^\eta_\epsilon,t^*+\sigma]$, since $\nabla_x F(t,\varphi(t))=0$, it holds
\begin{equation}\label{veps}
\epsilon \dot{v}_\epsilon(t) = \epsilon(\dot{u}_\epsilon(t)-\dot{\varphi}(t))= -\nabla_x F(t,u_\epsilon(t)) + \nabla_x F(t,\varphi(t)) -\epsilon\dot{\varphi}(t)\,.
\end{equation}
We define
\begin{equation*}
\hat{t}^\eta_\epsilon:=\inf\{t\in[t^\eta_\epsilon,t^*+\sigma]:\,\|v_\epsilon(t)\|\geq\eta\},
\end{equation*}
Note that, for $\epsilon$ small enough, by \eqref{smallnorm} we have that $\hat{t}^\eta_\epsilon>t^\eta_\epsilon$. By the definition of $\hat{t}^\eta_\epsilon$, applying, for fixed $t$, the mean-value theorem to the scalar-valued function $\tau\mapsto \langle \nabla_x F(t,\varphi(t)+\tau v_\epsilon(t)), v_\epsilon(t)\rangle$, and using \eqref{posdef}, we get 
\[
 \langle \nabla_x F(t,u_\epsilon(t)) - \nabla_x F(t,\varphi(t)), v_\epsilon(t)\rangle \geq \lambda\|v_\epsilon(t)\|^2\,.
\]
for all $t\in [t^\eta_\epsilon,\hat{t}^\eta_\epsilon]$.
From this and \eqref{veps}, it follows
\begin{equation*}
\begin{split}
\epsilon \frac{\mathrm{d}}{\mathrm{d}t}\frac{\|v_\epsilon(t)\|^2}{2}&=\epsilon \langle v_\epsilon(t), \dot{t}_\epsilon(t)\rangle= \langle -\nabla_x F(t,u_\epsilon(t)) + \nabla_x F(t,\varphi(t)) -\epsilon\dot{\varphi}(t), v_\epsilon(t)\rangle\\
                                                     &
\leq -\lambda\|v_\epsilon(t)\|^2+\frac{\epsilon}{2}\|\dot{\varphi}(t)\|^2+\frac{\epsilon}{2}\|v_\epsilon(t)\|^2\,.
\end{split}
\end{equation*}
Now, if $C_\varphi$ is an upper bound for $\|\dot{\varphi}(t)\|^2$ in $ [t^*-\sigma,t^*+\sigma]$ and we set $C_{\lambda,\epsilon}:=\left(-\frac{2\lambda}{\epsilon}+1\right)\to-\infty$ as $\epsilon\to0$, we finally get
\[
 \frac{\mathrm{d}}{\mathrm{d}t}\|v_\epsilon(t)\|^2\leq C_{\lambda,\epsilon}\|v_\epsilon(t)\|^2+C_\varphi.
\]
If $\hat{t}^\eta_\epsilon$ were smaller than $t^*+\sigma$, one must have $\frac{\mathrm{d}}{\mathrm{d}t}\|v_\epsilon(\hat{t}^\eta_\epsilon)\|^2 \ge 0$, while the above inequality  gives
\[
\frac{\mathrm{d}}{\mathrm{d}t}\|v_\epsilon(\hat{t}^\eta_\epsilon)\|^2\leq C_{\lambda,\epsilon}\eta^2+C_\varphi<0
\]
when $\epsilon$ is small enough. This  proves that $\hat{t}^\eta_\epsilon=t^*+\sigma$. With this, and since $t^\eta_\epsilon\to t^*$,  for an arbitrary $t\in (t^*, t^*+\sigma]$ we have that $\|v_\epsilon(t)\|\leq\eta$ when $\epsilon$ is small enough. By the arbitrariness of $\eta$, we get $v_\epsilon(t)\to 0$ for all $t\in (t^*, t^*+\sigma]$. Thus, $u(t)=\varphi(t)$ in $(t^*, t^*+\sigma]$, concluding the proof.
\endproof

\subsection{A simplified setting: the one-dimensional case}\label{one-dimensional}
We finally revisit the application of our results to the one-dimensional setting, that is for $X=\mathbb{R}$, where they are mostly well-known (see, e.g.\, \cite{Russi}), but stated under slightly different assumptions. 
We consider the singularly perturbed 1D-problem

\begin{equation}
\begin{cases}
\epsilon\dot{u}_\epsilon=\displaystyle-\frac{\mathrm{d}F}{\mathrm{d}x}(t,u_\epsilon(t)),\quad t\in[0,T],\\
\displaystyle\lim_{\epsilon\to0}u_\epsilon(0)=0\,,
\end{cases}
\label{1Dproblem}
\end{equation}
where $F:[0,T]\times \mathbb{R}\to \mathbb{R}$ is an energy satisfying {\bf(F0)-(F3)} and the assumption
\begin{align}\label{eq2_1d}
\frac{\mathrm{d}F}{\mathrm{d}x}(t,0)=0
\end{align}
for all $t\in [0,T]$. Setting
$
A(t):=\frac{\mathrm{d}^2 F}{\mathrm{d}x^2}(t,0),\,t\in[0,T] 
$,
\eqref{ipotesi_fond} reads in our case simply as
\begin{align}\label{ipfond2}
A(0)>0 \quad\hbox{and}\quad\int_{0}^T A(s)\,\mathrm{d}s<0\,.
\end{align}
The time $t^*$ is then defined as
\begin{equation}
t^*:=\min\left\{t\in(0,T):\,\int_{0}^tA(s)\,\mathrm{d}s=0\right\}\,.
\label{tstarbis}
\end{equation}
Assumption {\bf(A2)} is trivially satisfied in this setting, while {\bf (A3)} reduces to
\begin{align}\label{a31d} 
A(t^*)<0\,.
\end{align}

Under \eqref{eq2_1d}, \eqref{ipfond2} and \eqref{a31d}, the pointwise limit $u(t)$ of the solutions  $u_\epsilon$  to problem \eqref{1Dproblem} satisfies $u(t)=0$ for  $t\in[0,t^*)$ by Proposition \ref{propzero}, while a jump occurs at $t=t^*$, provided \eqref{decay} holds. Notice that \eqref{assume} is trivially satisfied in this setting. To analyse the behavior at $t^*$, we consider the smallest strictly positive and strictly negative critical point $u^{*,\pm}\in C(t^*)$ defined as
\begin{equation}\label{upm}
\begin{array}{c}
\displaystyle
u^{*,+}=\min\{v:\,v\in C(t^*),\, v>0\}\,,\\[5pt]
 u^{*,-}=\max  \{v:\,v\in C(t^*),\,v<0\}\,,
\end{array}
\end{equation}
respectively. Then, the following one-dimensional analog of Theorem \ref{jump2} holds. As one might expect, in the proof we can bypass the application of Proposition \ref{hetero-prel}: no previous knowledge about convergence of a rescaled version of the $u_\epsilon$ to a heteroclinic solution of \eqref{heterosystem} is actually needed. Notice also that \eqref{locmin1d} and \eqref{locmin21d} can be stated in a weaker way than \eqref{locmin1} and \eqref{locmin2}, respectively.
\begin{prop}
Assume that $X=\mathbb{R}$ and consider an energy $F$ satisfying {\bf(F0)-(F3)}, as well as \eqref{eq2_1d}. Let $(u_\epsilon)_\e$ be a sequence of solutions to \eqref{1Dproblem}, with initial data $u_\epsilon(0)$ as in \eqref{decay}.  Let $u:[0,T]\to \mathbb{R}$ be the pointwise limit  of $(u_\epsilon)_\epsilon$, as given by Theorem \ref{AR}. 
Assume that  \eqref{ipfond2} and \eqref{a31d} hold, for $t^*$ as in \eqref{tstarbis}, and define $u^{*,+}$ and $u^{*,-}$ as in \eqref{upm}.
Then:
\begin{itemize}
\item if 
\begin{equation}\label{locmin1d}
\frac{\mathrm{d}^2 }{\mathrm{d}x^2}F(t^*, u^{*,+}) \neq 0\quad\mbox{and}\quad u_\epsilon(0) >0\,,
\end{equation}
it holds $u_+(t^*)=u^{*,+}$. Furthermore, $u$ is of class $C^2$ in a right neighborhood of $t^*$.
\item if 
\begin{equation}\label{locmin21d}
\frac{\mathrm{d}^2 }{\mathrm{d}x^2}F(t^*, u^{*,-}) \neq 0\quad\mbox{and}\quad u_\epsilon(0) <0\,,
\end{equation}
it holds $u_+(t^*)=u^{*,-}$. Furthermore, $u$ is of class $C^2$ in a right neighborhood of $t^*$.
\end{itemize}
\end{prop}

\proof  
We prove the statement under assumption \eqref{locmin1d}, the proof of the other case being totally analogous. 
We begin by noticing that, since \eqref{a31d} gives $\frac{\mathrm{d}^2 }{\mathrm{d}x^2}F(t^*,0) < 0$,  from \eqref{upm} and \eqref{locmin1d} we immediately get
\[
\frac{\mathrm{d}^2 }{\mathrm{d}x^2}F(t^*, u^{*,+}) > 0\,.
\] 
We also observe that by Theorem \ref{AR} and {\bf(F0)}, it must hold $u_+(t^*)\in C(t^*)$. Since $u=0$ is a stationary solution of \eqref{1Dproblem}, we have $u_\epsilon(t)>0$ for all $t$, and thus $u(t)\ge 0$. Since $u(t)=0$ for  $t\in[0,t^*)$ by Proposition \ref{propzero}, while a jump occurs at $t=t^*$ by Theorem \ref{jump}, we have  $u_+(t^*)\ge u^{*,+}$. For an arbitrary $\eta>0$ we then set
\[
t^\eta_\epsilon:=\min\{t\in [0,T]: u_\epsilon(t)= u^{*,+}-\eta\}\,.
\]
Observe that, by the pointwise convergence of $u_\epsilon$ to $u$ and exploiting the continuity of the functions $u_\epsilon$,  the well-posedness of $t^\eta_\epsilon$ easily follows from the conditions $u_\epsilon(0)\to 0$ and $u_+(t^*)\ge u^{*,+}$.  Furthermore, from Proposition \ref{propzero} we have $\liminf_{\epsilon \to 0} t^\eta_\epsilon \ge t^*$.  Now, if $\limsup_{\epsilon \to 0} t^\eta_\epsilon > t^*$, we obtain $u_+(t^*)\le  u^{*,+}-\eta$, a contradiction. We therefore have found a sequence $t^\eta_\epsilon \to t^*$ with
\[
|u_\epsilon(t^\eta_\epsilon)- u^{*,+}|=\eta\,.
\]
We have thus established an analogous implication as \eqref{smallnorm}: with this, the same argument as in Theorem \ref{jump2}, with $u^{*,+}$ in place of $\overline{u}^*$, gives the desired conclusion.
\endproof

\end{document}